\documentclass[preprent]{elsarticle}
\usepackage{graphicx, epstopdf}
\usepackage{amssymb,amsmath}
\usepackage{amsbsy}
\usepackage{amsthm}
\usepackage{multirow}
\usepackage{indentfirst}
\usepackage{amscd}
\numberwithin{equation}{section}

\newtheorem{Theorem}{Theorem}[section]

\newtheorem{Assumption-Notation}[Theorem]{Assumption-Notation}

\newtheorem{Remark}{Remark}[section]

\newtheorem{Lemma}{Lemma}[section]

\usepackage{caption}
\long\def\symbolfootnote[#1]#2{\begingroup
\def\thefootnote{\fnsymbol{footnote}}\footnote[#1]{#2}\endgroup}

\journal{arXiv.org}
\begin{document}
\begin{frontmatter}
\title{Occupation times of generalized Ornstein-Uhlenbeck processes with two-sided exponential jumps}

 \author{Jiang Zhou}
 \ead{1101110056@pku.edu.cn}
 \author{Lan Wu \corref{cor1}}
 \ead{lwu@pku.edu.cn}
 \cortext[cor1]{Corresponding author.}
 \address{School of Mathematical Sciences, Peking University, Beijing 100871, PR China}

\begin{abstract}{For an Ornstein-Uhlenbeck process driven by a double exponential jump diffusion process, we obtain formulas for the joint Laplace transform of it and its occupation times.
The approach used is remarkable and can be extended to investigate the occupation times of an Ornstein-Uhlenbeck process driven by a more general L\'evy process. }
\end{abstract}

\end{frontmatter}
\section{Introduction}
A generalized Ornstein-Uhlenbeck process $X=(X_t)_{t \geq 0}$ is characterized by the following  equation:
\begin{equation}
d X_t = \kappa (\alpha- X_t)dt+ d L_t,\ \ t>0,
\end{equation}
where $\kappa > 0$,  $\alpha \in \mathbb R$ and $X_0=x$ is non-random; $L=(L_t)_{t \geq 0}$ is a L\'evy process. The first passage time of $X$ has been investigated considerably, the reader is referred to Jacobsen and Jensen (2007) and Borovkov and Novikov (2008) and literatures therein for the details.

In this article, we are interested in the joint Laplace transform of $X$ and its occupation times, i.e.,
$E \left[e^{- p \int_{0}^{T}\textbf{1}_{\{X_t \leq b \}}dt + q X_T}\right]$,
where $p > 0$, $q$ is some suitable constant and for a given set $A$, $\textbf{1}_{A}$ is the indicator function; and the objection is deriving formulas for its Laplace transform, i.e.,
\begin{equation}
\int_{0}^{\infty} e^{- s T} E \left[e^{- p \int_{0}^{T}\textbf{1}_{\{X_t \leq b \}}dt+q X_T}\right]dT, \ \ for \ \ s > 0.
\end{equation}
If $L_t$ in (1.1) is a Brownian motion with drift, formulas for (1.2) with $q = 0$ are known, one can refer to Li and Zhou (2013) for example. Thus, we focus on the case that $L_t$ is a jump diffusion process, and to the best of our knowledge, we are the first to investigate (1.2) under the assumption that $L_t$ has jumps.

Results obtained here can be applied to price occupation time derivatives as in Cai et al. (2010), in which the authors have noted that there are several products in the real market with payoffs depending on the occupation times of an interest rate or a spread of swap rates (see Remark 3.3 in that paper). Usually, interest rates or spreads of swap rates are modeled by generalized Ornstein-Uhlenbeck processes since they are mean reversion. Therefore, our results are very important for pricing such derivatives.

The remainder of the paper is organized as follows. Section 2 presents the details of our model and some important preliminary results, and  Section 3 derives the main results.

\section{Details of the model and some preliminary outcomes}
\subsection{The model}
In this paper, the L\'evy process $L=(L_t)_{t \geq 0}$ in (1.1) is assumed to be a double exponential jump diffusion process, i.e.,
\begin{equation}
L_t = \mu t+\sigma W_t + \sum_{k=1}^{N_t}Y_k.
\end{equation}
where $\mu \in \mathbb R$ and $\sigma > 0$ are constants; $W=(W_t)_{t\geq 0}$ is a standard Brownian motion and is independent of $\sum_{k=1}^{N_t}Y_k$, which is a compound Poisson process with intensity $\lambda$ and the following jump distribution:
\begin{equation}
f_{Y}(y) = P\left(Y_1 \in dy\right)/dy=p \eta e^{-\eta y} \textbf{1}_{\{y > 0\}} + (1-p) \vartheta e^{\vartheta y}\textbf{1}_{\{y < 0\}},
\end{equation}
where $0< p <1$ and $\eta_, \vartheta > 0$. For the investigation on such a L\'evy process $L_t$, we refer to Kou and Wang (2003).

In what follows, for given $x \in \mathbb R$, $\mathbb P_x$ is the law of $X$ starting from $x$ with
$\mathbb E_x$ denoting the matching expectation. If $x=0$, we drop the subscript and write simply $\mathbb P$ and $\mathbb E$. Recall (1.1) and note that
\[
X_t=\alpha + e^{-\kappa t} (X_0-\alpha)+\int_{0}^{t}e^{\kappa (s-t)}dL_s, \ \ t \geq 0,
\]
and
\[
\varphi(\theta):=\ln\left(\mathbb E\left[e^{\theta L_1}\right]\right)=\mu \theta +\frac{\sigma^2}{2} \theta^2 +\lambda\left(\frac{p\eta}{\eta-\theta}+\frac{(1-p)\vartheta}{\vartheta + z}-1\right),\ \ -\vartheta < \theta < \eta.
\]
It is known that (see Theorem 1 in Lukacs (1969))
\[
\mathbb E\left[e^{\int_{0}^{t}\hat{f}(s)dL_s}\right]=e^{\int_{0}^{t}\varphi(\hat{f}(s))ds},
\]
where $\hat{f}(s)$ is a non-random function. Thus, for $-\vartheta < \theta < \eta$, we can derive
\begin{equation}
\begin{split}
&\ln\left(\mathbb E_x\left[e^{\theta X_t}\right]\right)
=\theta e^{-\kappa t}x + \theta \alpha(1-e^{-\kappa t})+
\mu \theta \frac{1-e^{-\kappa t}}{\kappa}\\
&+\frac{\sigma^2(1-e^{-2\kappa t})}{4 \kappa}\theta^2
+\lambda\left(\frac{p}{\kappa}\ln\Big(\frac{\eta-\theta e^{-\kappa t}}{\eta - \theta}\Big)+
\frac{(1-p)}{\kappa}\ln\Big(\frac{\vartheta+\theta e^{-\kappa t}}{\vartheta+\theta}\Big)\right),
\end{split}
\end{equation}
where in the above derivation, we have used the following identity:
\[
t=\frac{p}{\kappa}\ln\Big( e^{\kappa t}\Big)+
\frac{(1-p)}{\kappa}\ln\Big(e^{\kappa t}\Big).
\]
For given $-\vartheta < \theta < \eta$, note that
\begin{equation}
\lim_{t \uparrow \infty} \ln\left(\mathbb E_x\left[e^{\theta X_t}\right]\right)
=\theta \alpha+
\frac{\mu \theta}{\kappa}+\frac{\sigma^2}{4 \kappa}\theta^2
+\lambda\Big(\frac{p}{\kappa}\ln\big(\frac{\eta}{\eta - \theta}\big)+
\frac{(1-p)}{\kappa}\ln\big(\frac{\vartheta}{\vartheta+\theta}\big)\Big).
\end{equation}

For the process $X$ given by (1.1) and (2.1), the purpose of the paper is to deduce formulas for (1.2), i.e.,
\begin{equation}
\int_{0}^{\infty} e^{- s T}\mathbb E_x\left[e^{- p \int_{0}^{T}\textbf{1}_{\{X_t \leq b \}}dt+q X_T}\right]dT, \ \ for \ \ s > 0.
\end{equation}
Our approach depends on results about the one-sided exit problem of $X$, which will be presented in the next subsection.
\begin{Remark}
In Cai et al. (2010), they have obtained expressions for (2.5) under the assumption that the process $X$ is a double exponential jump diffusion process (i.e., $X_t=L_t$ or $\kappa =0$ in (1.1)). A contribution here is extending their results to the case of $\kappa >0$.
\end{Remark}

\begin{Remark}
The method in this article can be extended to calculate (2.5) when $X_t$ is given by (1.1) and $L_t$ is a hyper-exponential jump diffusion process\footnote{In other words, the distribution of $Y_1$ in (2.1) is generalized to the following form:
\[f_{Y}(y) = \sum_{i=1}^{m}p_i\eta_i e^{-\eta_i y}\textbf{1}_{\{y > 0\}} + \sum_{j=1}^{n}q_j\vartheta_j e^{\vartheta_j y}\textbf{1}_{\{y < 0\}}.\]}.
But, to illustrate the ideas in our approach clearly, it is desirable to consider a simper model.
\end{Remark}

\subsection{Results on the one-sided exit problem of $X$}
First of all, for $a, c \in \mathbb R$, define
\begin{equation}
\tau_a^- := \inf\{t > 0: X_t \leq a\}, \ \ \ \tau_c^+ := \inf\{t > 0: X_t \geq c \}.
\end{equation}
For the stopping time $\tau_a^-$, $q > 0$ and $\xi \geq 0$, we want to compute the following quantities:
\begin{equation}
\mathbb E_x\left[e^{- q \tau_a^- - \xi (a-X_{\tau_a^-})}\rm{\bf{1}}_{\{X_{\tau_a^-} < a\}}\right] \ \ and \ \  \mathbb E_x\left[e^{- q \tau_a^-}\rm{\bf{1}}_{\{X_{\tau_a^-} = a\}}\right], \ \ for \ \  x > a.
\end{equation}

If $\sigma=\mu=\alpha=0$ in (1.1) and (2.1), formulas for (2.7) have been derived in Jacobsen and Jensen (2007)(see Proposition 5), in which they also considered the case of $\sigma > 0$, $\alpha \in \mathbb R$ and $\mu=0$. Applying similar ideas in Jacobsen and Jensen (2007) can lead to the following Lemma 2.1 for the case of $\mu \in \mathbb R$. Before giving Lemma 2.1, we introduce some notations.

For $\alpha$ given in (1.1), $q>0$ and $x \in \mathbb R$, define
\begin{equation}
\hat{x}:= x - \alpha \ \ and \ \
\psi_q(x):=x^{\frac{q}{\kappa}-1}e^{-\frac{\sigma^2}{4 \kappa}x^2+\frac{\mu}{\kappa}x}(\eta+x)^{\frac{\lambda p}{\kappa}}(x-\vartheta)^{\frac{\lambda (1-p)}{\kappa}},
\end{equation}
where $\psi_q(x)$ is treated as a complex valued function when $x < \vartheta$.
And for $\xi, \rho \geq 0$, $q > 0$, $x \in \mathbb R$ and $i=1,2,3,4$,
\begin{equation}
\begin{split}
&F^q_i(x):=\int_{\Gamma _i} |\psi_q(z)|e^{-xz} dz, \\
&D^{\rho,q}_i(x):=\int_{\Gamma_i}\frac{\eta +\rho}{z+\eta}|\psi_q(z)|e^{- x z} dz,\\
&C^{\xi,q}_i(x):=-\int_{\Gamma _i}\frac{\vartheta +\xi}{z-\vartheta}|\psi_q(z)|e^{- x z} dz,
\end{split}
\end{equation}
where $\Gamma_1=(0,\vartheta)$, $\Gamma_2=(\vartheta,\infty)$, $\Gamma_3=(-\eta,0)$ and  $\Gamma_4=(-\infty,-\eta)$; $|\psi_q(z)|$ is the module of $\psi_q(z)$.

\begin{Lemma}
For $x > a$, $q > 0$ and $\xi \geq 0$,
\begin{equation}
\left(\begin{array}{cc}
\mathbb E_x\left[e^{-q \tau_a^-}\rm{\bf{1}}_{\{X_{\tau_a^- = a }\}}\right]\\
\mathbb E_x\left[e^{-q \tau_a^- -\xi(a-X_{\tau_a^-})}\rm{\bf{1}}_{\{X_{\tau_a^- < a }\}}\right]
\end{array}\right)=\left(\begin{array}{ccc}
F_1^{q}(\hat{a})&C_1^{\xi,q}(\hat{a})\\
F_2^{q}(\hat{a})&C_2^{\xi,q}(\hat{a})\\
\end{array}\right)^{-1}\left(\begin{array}{cc}
F^q_1(\hat{x})\\
F^q_2(\hat{x})\\
\end{array}\right).
\end{equation}
\end{Lemma}

\begin{proof}
The derivation depends on some similar ideas and calculations in the proof of Proposition 5 in Jacobsen and Jensen (2007).

(i) Assume $\alpha =0$, then $x=\hat{x}$ (recall (2.8)) for any $x \in \mathbb R$. Consider a function $f_1(x)$ defined as follows:
\begin{equation}
f_1(x)=\left\{\begin{array}{cc}
F_1^q(x), & x \geq a,\\
C^{\xi,q}_1(a) e^{-\xi (a-x)}, & x < a.
\end{array}\right.
\end{equation}
Obviously, $f_1(x)$ is bounded and differentiable on $[a,\infty)$.

Let $\mathcal{A}$ be the generator of $X$. In other words, for $x \geq a$,
\[
\mathcal{A} f_1(x)=-\kappa x f_1^{\prime}(x)+\mu f_1^{\prime}(x)+\frac{\sigma^2}{2}f_1^{\prime \prime}(x)
+\lambda \int_{-\infty}^{\infty}f_1(x+y)f_Y(y)dy-\lambda f_1(x),
\]
where $f_Y(y)$ given by (2.2).
Then, for $x \geq a$, some straightforward calculations will lead to
\begin{equation}
\begin{split}
&\mathcal{A} f_1(x)- q f_1(x)= \kappa x\int_{\Gamma_1}z|\psi_q(z)|e^{-xz}dz - \mu \int_{\Gamma_1}z|\psi_q(z)|e^{-xz}dz
\\
&+\frac{\sigma^2}{2}\int_{\Gamma_1}z^2|\psi_q(z)|e^{-xz}dz
+\lambda p \int_{\Gamma_1}\frac{\eta}{\eta+z}|\psi_q(z)|
e^{-x z}dz\\
&+ \lambda (1-p)\vartheta \int_{\Gamma_1}\frac{1-e^{(\vartheta-z)(a-x)}}{\vartheta-z}|\psi_q(z)|e^{-xz}dz
-\lambda p f_1(x)\\
&-\lambda (1-p)f_1(x) - q f_1(x)+C^{\xi,q}_1(a) \lambda (1-p)\frac{\vartheta}{\vartheta+\xi}e^{\vartheta(a-x)},
\end{split}
\end{equation}
where the fourth and fifth integral on the right-hand side of (2.12) follows from exchanging the order of integration.

Recall $\Gamma_1=(0,\vartheta)$. For the first integral on the right-hand side of (2.12), applying partial integration yields
\begin{equation}
\kappa x\int_{0}^{\vartheta}z|\psi_q(z)|e^{-xz}dz=\kappa \int_{0}^{\vartheta}\left(\frac{\partial |\psi_q(z)|}{\partial z}z+|\psi_q(z)|\right)e^{-xz}dz,
\end{equation}
where we have used that
\[
\lim_{z \downarrow 0}z|\psi_q(z)|=0 \ \ and  \ \ \lim_{z\uparrow \vartheta}z|\psi_q(z)|=0.
\]
From (2.9), (2.11) and (2.13),  we can write the right-hand side of (2.12) as
\[
\int_{\Gamma_1}e^{-xz}\Big(\kappa \frac{\partial|\psi_q(z)|}{\partial z}z +|\psi_q(z)|\big(\kappa+\frac{\sigma^2}{2}z^2
-\mu z- \frac{\lambda p z}{\eta+z}+\frac{\lambda(1-p) z}{\vartheta-z}-q\big)\Big)dz,
\]
which combined with the definition of $\psi_q(z)$ in (2.8), produces\footnote{In fact, the expression of $|\psi_q(z)|$ for $0< z < \vartheta$ is obtained by solving the equation:
\[\kappa \frac{\partial|\psi_q(z)|}{\partial z}z +|\psi_q (z)|\left(\kappa+\frac{\sigma^2}{2}z^2
-\mu z-\lambda p \frac{z}{\eta+z}+\lambda(1-p)\frac{z}{\vartheta-z}-q\right)=0.\] }
\begin{equation}
\mathcal{A} f_1(x)- q f_1(x)=0, \ \ for \ \ x \geq a.
\end{equation}

From (2.14), It$\hat{o}$'s formula and the dominated convergence theorem will give
\begin{equation}
\lim_{t \uparrow \infty}\mathbb E_x\left[e^{-q (\tau_a^- \wedge t)} f_1(X_{\tau_a^-\wedge t})\right]=\mathbb E_x\left[e^{-q \tau_a^-} f_1(X_{\tau_a^-})\right]=f_1(x), \ \ x \geq a.
\end{equation}
It follows from (2.15) and the definition of $f_1(x)$ for $x < a$ in (2.11) that
\begin{equation}
C_1^{\xi,q}(a) \mathbb E_x\left[e^{-q \tau_a^- - \xi (a-X_{\tau_a^-})}\textbf{1}_{\{X_{\tau_a^-} < a\}}\right]+
\mathbb E_x\left[e^{-q \tau_a^- }\textbf{1}_{\{X_{\tau_a^-} = a\}}\right]f_1(a)=f_1(x),  \ \ x \geq a.
\end{equation}

Similarly, define
\begin{equation}
f_2(x)=\left\{\begin{array}{cc}
F_2^{q}(x), & x \geq a,\\
C^{\xi,q}_2(a) e^{-\xi (a-x)}, & x < a,
\end{array}\right.
\end{equation}
where $F_2^{q}(x)$ and $C^{\xi,q}_2(a)$ are given by (2.9). Note that $\Gamma_2=(\vartheta,\infty)$ and $|\psi_q(z)|=\psi_q(z)$ for $z > \vartheta$. In addition,
\[
\lim_{z \downarrow \vartheta}z\psi_q(z)=0 \ \ and  \ \ \lim_{z\uparrow \infty}z \psi_q(z)=0,
\]
and for $z > \vartheta$,
\[\kappa \psi_q^{\prime}(z)z +\psi_q (z)\left(\kappa+\frac{\sigma^2}{2}z^2
-\mu z-\lambda p \frac{z}{\eta+z}+\lambda(1-p)\frac{z}{\vartheta-z}-q\right)=0.
\]
Thus, for $x \geq a$, it can be proved that
\begin{equation}
C_2^{\xi,q}(a) \mathbb E_x\left[e^{-q \tau_a^- - \xi (a-X_{\tau_a^-})}\textbf{1}_{\{X_{\tau_a^-} < a\}}\right]+
\mathbb E_x\left[e^{-q \tau_a^- }\textbf{1}_{\{X_{\tau_a^-} = a\}}\right]f_2(a)=f_2(x).
\end{equation}

Therefore, formula (2.10) for $\alpha =0$ is derived from (2.11) and (2.16)--(2.18).

(ii) If $\alpha \neq 0$, let $\hat{X}_t=X_t-\alpha$ for $t \geq 0$ and note that
\[
d \hat{X}_t= - \kappa \hat{X}_t dt+ d L_t.
\]
This yields the desired result.
\end{proof}

\begin{Remark}
Lemma 2.1 implies that
\begin{equation}
\mathbb E_x\left[e^{-q \tau_a^- -\xi(a-X_{\tau_a^-})}\rm{\bf{1}}_{\{X_{\tau_a^- < a }\}}\right]=\frac{F^q_1(\hat{a})F^q_2(\hat{x})-F^q_2(\hat{a})F^q_1(\hat{x})}
{C^{\xi,q}_2(\hat{a}) F^q_1(\hat{a})-C^{\xi,q}_1(\hat{a}) F^q_2(\hat{a})}.
\end{equation}
Recall (2.9) and note that the right-hand side of (2.19) can be written as $\hat{F}(\hat{a},\hat{x})/(\vartheta + \xi)$, where $\hat{F}(\hat{a},\hat{x})$ is not dependent on $\xi$. So formula (2.19) confirms the following well-known result:
 \[\mathbb E_x\left[e^{-q \tau_a^-}\rm{\bf{1}}_{\{a- X_{\tau_a^-} \in dy \}}\right]=
 \mathbb E_x\left[e^{-q \tau_a^-}\rm{\bf{1}}_{\{X_{\tau_a^-} < a \}}\right]\vartheta e^{-\vartheta y} dy, \ \ for \ \ y > 0,
 \]
which is due to the lack of memory of exponential distributions
\end{Remark}

For the stopping time $\tau_c^+$ in (2.6),  similar results to Lemma 2.1 hold.
\begin{Lemma}
For $x < c$, $q >0$ and $\rho \geq 0$,
\begin{equation}
\left(\begin{array}{cc}
\mathbb E_x\left[e^{-q \tau_c^+}\rm{\bf{1}}_{\{X_{\tau_c^+ = c}\}}\right]\\
\mathbb E_x\left[e^{-q \tau_c^+ -\rho (X_{\tau_c^+ - c})}\rm{\bf{1}}_{\{X_{\tau_c^+ > c }\}}\right]
\end{array}\right)=\left(\begin{array}{ccc}
F_3^{q}(\hat{c})& D_3^{\rho,q}(\hat{c})\\
F_4^{q}(\hat{c})& D_4^{\rho,q}(\hat{c})\\
\end{array}\right)^{-1}\left(\begin{array}{cc}
F_3^{q}(\hat{x})\\
F_4^{q}(\hat{x})\\
\end{array}\right).
\end{equation}
\end{Lemma}

\begin{proof}
From the derivation of Lemma 2.1, it is enough to consider the case of $\alpha=0$. Thus we assume that $x=\hat{x}$ in this proof. For $i=1,2$, consider the following function
\begin{equation}
g_i(x)=\left\{\begin{array}{cc}
F_{2+i}^q(x), & x \leq  c,\\
D^{\rho,q}_{2+i}(c) e^{-\rho(x-c)}, & x > c,
\end{array}\right.
\end{equation}
where $F_{2+i}^q(x)$ and $D^{\rho,q}_{2+i}(c)$ are given by (2.9).
Similar to the derivation of (2.15), it can be shown that
\begin{equation}
\mathbb E_x\left[e^{-q \tau_c^+} g_i (X_{\tau_c^+})\right]=g_i(x), \ \ for \ \ x \leq c.
\end{equation}
This result and the definition of $g_i(x)$ for $x > c$ in (2.21) give us
\[
D_{2+i}^{\rho,q}(c) \mathbb E_x\left[e^{-q \tau_c^+ - \rho (X_{\tau_c^+}-c)}\textbf{1}_{\{X_{\tau_c^+} > c\}}\right]+
\mathbb E_x\left[e^{-q \tau_c^+}\textbf{1}_{\{X_{\tau_c^+} = c\}}\right]g_i(c)=g_i(x),  \ \ x \leq c,
\]
from which (2.20) is deduced.
\end{proof}

Similar to Remark 2.3, it also holds that
 \[\mathbb E_x\left[e^{-q \tau_c^+}\rm{\bf{1}}_{\{ X_{\tau_c^+} - c \in dy \}}\right]=
 \mathbb E_x\left[e^{-q \tau_c^+}\rm{\bf{1}}_{\{X_{\tau_c^+} > c \}}\right]\eta e^{-\eta y} dy, \ \ for \ \  y > 0.
 \]
In particular, we have the following lemma.

\begin{Lemma}
(i) For any nonnegative measurable function $f(x)$ on $\mathbb R$ such that $\int_{-\infty}^{0}f(a+y)e^{\vartheta y}dy < \infty$, $q>0$ and $x > a$, it holds that
\begin{equation}
\begin{split}
\mathbb E_x\left[e^{-q\tau_a^-}f(X_{\tau_a^-})\right]
&=\mathbb E_x\left[e^{-q\tau_a^-}\rm{\bf{1}}_{\{X_{\tau_a^-} < a\}}\right]\int_{-\infty}^{0}f(a + y)\vartheta e^{\vartheta y}dy\\
&\ \ + f(a)\mathbb E_x\left[e^{-q\tau_a^-}\rm{\bf{1}}_{\{X_{\tau_a^-} = a\}}\right].
\end{split}
\end{equation}
(ii) For any nonnegative measurable function $f(x)$ on $\mathbb R$ such that $\int_{0}^{\infty}f(c+y)e^{-\eta y}dy < \infty$, $q > 0$ and $x < c$, it holds that
\begin{equation}
\begin{split}
\mathbb E_x\left[e^{-q\tau_c^+}f(X_{\tau_c^+})\right]
&=\mathbb E_x\left[e^{-q\tau_c^+}\rm{\bf{1}}_{\{X_{\tau_c^+} > c\}}\right]\int_{0}^{\infty}f(c + y)\eta e^{-\eta y}dy\\
&\ \ +f(c) \mathbb E_x\left[e^{-q\tau_c^{+}}\rm{\bf{1}}_{\{X_{\tau_c^+}=c\}}\right].
\end{split}
\end{equation}
\end{Lemma}

\section{Main results}
In this section, assume that $b \in \mathbb R$, $s > 0$, $p > -s$ and $-\vartheta < 3\cdot q < \eta $. The objection is to deduce the expression of
\begin{equation}
\begin{split}
V(x):
&=\int_{0}^{\infty}se^{-s T}\mathbb E_x\left[e^{-p \int_{0}^{T}\textbf{1}_{\{X_t \leq b\}}dt+q X_T}\right]dT\\
&=\mathbb E_x\left[e^{-p \int_{0}^{e(s)}\textbf{1}_{\{X_t \leq b\}}dt+q X_{e(s)}}\right],
\end{split}
\end{equation}
where the variable $e(s)$, independent of $X$, is an exponential distribution with parameter $s$. Recall (2.3) and (2.4). Since $-\vartheta < 3\cdot q < \eta $, we have
\[\mathbb E_x\left[e^{ q X_{e(s)}}\right]=s\int_{0}^{\infty}e^{-st}\mathbb E_x\left[e^{ q X_{t}}\right]dt < \infty, \]
for any given $s > 0$ and $x \in \mathbb R$. For $x \in \mathbb R$, define
\begin{equation}
\begin{split}
&V_1^s(x):=\mathbb E_x\left[e^{ q X_{e(s)}}\right], \\
&T_{\eta}^{s}(x)=\int_{0}^{\infty}V_1^{s}(x+z)\eta e^{-\eta z}dz,\\  &T_{\vartheta}^{s}(x)=\int_{-\infty}^{0}V_1^{s}(x+z)\vartheta e^{\vartheta z}dz.
\end{split}
\end{equation}
Due to (2.3), $V_1^s(x)$, $T_{\eta}^{s}(x)$ and $T_{\vartheta}^{s}(x)$ are considered as known functions from now on.

The main results are given in Theorem 3.1, and for its derivation, we improve the approach in Wu and Zhou (2016). Especially, the technic used in proving $V^{\prime}(b-)=V^{\prime}(b+)$ in Lemma 3.1 (will be presented after the proof of Theorem 3.1) is new and novel, and is expected to give some motivations to the investigation on the occupation times of Ornstein-Uhlenbeck processes driven by more general L\'evy processes and other stochastic processes.

\begin{Theorem}
For $s > 0$, $p > -s$, $b \in \mathbb R$ and $-\vartheta < 3 \cdot q < \eta$, we have
\begin{equation}
\begin{split}
&\mathbb E_x\left[e^{-p \int_{0}^{e(s)}\rm{\bf{1}}_{\{X_t \leq b\}}dt+q X_{e(s)}}\right]=
\left\{\begin{array}{cc}
\frac{s}{k}V_1^{k}(x)+\sum_{i=1}^{2}J_i F_{2+i}^{k}(\hat{x}),&  x \leq b,\\
V_1^s(x)+\sum_{i=1}^{2}N_i F_i^s(\hat{x}), & x \geq b,
\end{array}\right.
\end{split}
\end{equation}
where $k=s+p$ and the constants $J_i$ and $N_i$ satisfy
\begin{equation}
\left(J_1,J_2,N_1,N_2\right)\mathbb Q = w ,
\end{equation}
with
\[
w=\left(V^s_1(b)-\frac{s V_1^{k}(b)}{k},V_1^{ s \prime}(b)- \frac{s V_1^{k \prime}(b)}{k},T_{\eta}^s(b)-\frac{s}{k}T_{\eta}^{k}(b),
T_{\vartheta}^s(b)-\frac{s}{k}T_{\vartheta}^{k}(b)\right),
\]
and
\begin{equation}
\mathbb Q=\left(\begin{array}{ccccc}
F_3^{k}(\hat{b})&F_3^{k \prime}(\hat{b})&D_3^{0,k}(\hat{b})&C_3^{0,k}(\hat{b})\\
F_4^{k}(\hat{b})&F_4^{k \prime}(\hat{b})&D_4^{0,k}(\hat{b})&C_4^{0,k}(\hat{b})\\
-F_1^{s}(\hat{b})&-F_1^{s \prime}(\hat{b})&-D_1^{0,s}(\hat{b})&-C_1^{0,s}(\hat{b})\\
-F_2^{s}(\hat{b})&-F_2^{s \prime}(\hat{b})&-D_2^{0,s}(\hat{b})&-C_2^{0,s}(\hat{b})
\end{array}\right).
\end{equation}
Here, in (3.3) and (3.5), for any given $r > 0$, $F_i^{r}(\hat{b})$, $D_i^{0,r}(\hat{b})$ and $C_i^{0,r}(\hat{b})$ are given by (2.9).
\end{Theorem}

\begin{proof}
In this derivation, some similar ideas in Wu and Zhou (2016) will be used.
First, we know from (3.1) that
\begin{equation}
\begin{split}
V(x)\leq \left\{\begin{array}{cc}
\mathbb E_x\left[e^{q X_{e(s)}}\right],& if \ \ p\geq 0,\\
 \mathbb E_x\left[e^{-pe(s)+q X_{e(s)}}\right], & if - s < p < 0,
\end{array}\right.
\end{split}
\end{equation}
and note that
\begin{equation}
 \mathbb E_x\left[e^{-pe(s)+q X_{e(s)}}\right]=s\int_{0}^{\infty}e^{-s t} \mathbb E_x\left[e^{-pt+q X_{t}}\right]dt=\frac{s}{s+p} \mathbb E_x\left[e^{q X_{e(p+s)}}\right].
\end{equation}

For $x < b$, it follows from the lack of memory property of $e(s)$ and the strong Markov property of $X$ that (recall (3.1))
\begin{equation}
\begin{split}
V(x)
&=\mathbb E_x\left[e^{-p\tau_b^+-p\int_{\tau_b^+}^{e(s)}\textbf{1}_{\{X_t \leq b\}}dt+q X_{e(s)}}\textbf{1}_{\{e(s) > \tau_b^+\}}\right]\\
&+\mathbb E_x\left[e^{-pe(s)+q X_{e(s)}}\textbf{1}_{\{e(s) \leq \tau_b^+\}}\right]\\
&=\mathbb E_x\left[e^{-\kappa \tau_b^+}V(X_{\tau^+_b})\right]+ \frac{s}{\kappa}\mathbb E_x\left[e^{q X_{e(\kappa)}}\textbf{1}_{\{e(\kappa) \leq \tau_b^+\}}\right]\\
&=\mathbb E_x\left[e^{-\kappa \tau_b^+}\left(V(X_{\tau^+_b})-\frac{s}{\kappa}V_1^{\kappa}(X_{\tau^+_b})\right)\right]
+\frac{s}{\kappa}V_1^{\kappa}(x)\\
&=\frac{s}{\kappa}V_1^{\kappa}(x)+\sum_{i=1}^{2}J_i F_{2+i}^{\kappa}(\hat{x}),
\end{split}
\end{equation}
where $k =s+p$ and $V_1^{k}(x)$ is given by (3.2); the final equality follows from (2.20), (2.24) and (3.2) with
\begin{equation}
\left(J_1,J_2\right)=\Big(V(b)- \frac{s}{k}V_1^{k}(b),\int_{0}^{\infty}V(b+z)\eta e^{-\eta z}dz-\frac{s}{k}T_{\eta}^{k}(b)\Big)\left(\begin{array}{ccc}
F_3^{k}(\hat{b})& D_3^{0,k}(\hat{b})\\
F_4^{k}(\hat{b})& D_4^{0,k}(\hat{b})\\
\end{array}\right)^{-1}.
\end{equation}
Similarly, for $x > b$, formulas (2.10), (2.23) and (3.2) lead to
\begin{equation}
\begin{split}
&V(x)
=\mathbb E_x\left[e^{-p\int_{0}^{e(s)}\textbf{1}_{\{X_t \leq b\}}dt+qX_{e(s)}}\textbf{1}_{\{e(s) \leq \tau_b^-\}}\right]\\
&\ \ \ \ \ \ +\mathbb E_x\Big[e^{-p\int_{\tau_b^-}^{e(s)}\textbf{1}_{\{X_t \leq b\}}dt+qX_{e(s)}}\textbf{1}_{\{e(s) > \tau_b^-\}}\Big]\\
&=\mathbb E_x\left[e^{qX_{e(s)}}\textbf{1}_{\{e(s) \leq \tau_b^-\}}\right]+\mathbb E_x\left[
e^{-s\tau_b^-}V(X_{\tau^-_b})\right]\\
&=\mathbb E_x\left[e^{qX_{e(s)}}\right]-
\mathbb E_x\left[e^{qX_{e(s)}}\textbf{1}_{\{e(s) > \tau_b^-\}}\right]
+\mathbb E_x\left[e^{-s\tau_b^-}V(X_{\tau^-_b})\right]\\
&=V^s_1(x)+\mathbb E_x\left[e^{-s\tau_b^-}\big(V(X_{\tau^-_b})-V^s_1(X_{\tau^-_b})\big)\right]
=V_1^s(x)+ \sum_{i=1}^{2}N_i F_i^s(\hat{x}),
\end{split}
\end{equation}
where
\begin{equation}
\left(N_1,N_2\right)\left(\begin{array}{ccc}
F_1^{s}(\hat{b})&C_1^{0,s}(\hat{b})\\
F_2^{s}(\hat{b})&C_2^{0,s}(\hat{b})\\
\end{array}\right)=\left(V(b)-V^s_1(b),\int_{-\infty}^{0}V(b+z)\vartheta e^{\vartheta z}dz-T_{\vartheta}^s(b)\right).
\end{equation}

Formulas (3.9) and (3.11) imply
\begin{equation}
V(b)=\frac{s}{k}V_1^{k}(b)+\sum_{i=1}^{2}J_i F_{2+i}^{k}(\hat{b})=\sum_{i=1}^{2}N_i F_i^s(\hat{b})+V^s_1(b).
\end{equation}
Besides, we know $V^{\prime}(b-)=V^{\prime}(b+)$ (see Lemma 3.1), which combined with (3.8) and (3.10), leads to
\begin{equation}
\sum_{i=1}^{2}J_i \frac{\partial }{\partial \hat{x}}\Big(F_{2+i}^{k}(\hat{x})\Big)_{\hat{x}=\hat{b}}
+\frac{s}{k}V_1^{k \prime}(b)=\sum_{i=1}^{2}N_i \frac{\partial }{\partial \hat{x}}\Big(F_{i}^{s}(\hat{x})\Big)_{\hat{x}=\hat{b}}+V_1^{s \prime}(b).
\end{equation}

From (2.9), (3.2) and (3.10), we can derive
\[
\begin{split}
\int_{0}^{\infty}V(b+z)\eta e^{-\eta z}dz
&=\int_{0}^{\infty}V^s_1(b+z)\eta e^{-\eta z}dz+\sum_{i=1}^{2}N_i \int_{0}^{\infty}F_i^s(\hat{b}+z)\eta e^{-\eta z}dz\\
&=T_{\eta}^s(b)+\sum_{i=1}^{2}N_iD_i^{0,s}(\hat{b}),
\end{split}
\]
this result and formula (3.9) mean
\begin{equation}
\sum_{i=1}^{2}J_iD_{2+i}^{0,k}(\hat{b})+\frac{s}{k}T_{\eta}^{k}(b)=T_{\eta}^s(b)+\sum_{i=1}^{2}N_iD_i^{0,s}(\hat{b}).
\end{equation}
Applying similar derivations to (3.8) and (3.11) and using (3.2), we have
\begin{equation}
\begin{split}
&\sum_{i=1}^{2}N_i C_i^{0,s}(\hat{b})+T_{\vartheta}^s(b) =\int_{-\infty}^{0}V(b+z)\vartheta e^{\vartheta z}dz\\
&=\sum_{j=1}^{2}J_i\int_{-\infty}^{0}F_{2+i}^{k}(\hat{b}+z)\vartheta e^{\vartheta z}dz+\frac{s}{k}
T_{\vartheta}^{k}(b)
=\sum_{j=1}^{2}J_iC_{2+i}^{0,k}(\hat{b})+\frac{s}{k}
T_{\vartheta}^{k}(b),
\end{split}
\end{equation}
where the last equality is due to (2.9).
So, (3.3) is derived from (3.12)--(3.15) and the proof is completed.
\end{proof}

\begin{Lemma}
For the function $V(x)$, defined in (3.1), its derivative is continuous at $b$, i.e., $V^{\prime}(b-)=V^{\prime}(b+)$.
\end{Lemma}
\begin{proof}
(i) Define the continuous component of $X$ as $X^c$, i.e., $X_t^c=\alpha + e^{-\kappa t} (X_0-\alpha)+\int_{0}^{t}e^{\kappa (s-t)}dL^c_s$, where $L^c_t=\mu t +\sigma W_t$ for $t\geq 0$, and introduce the stopping time $\tau^c_{b,\varepsilon}:=\inf\{t > 0: X^c_t > b+\varepsilon \ \ or \ \ X^c_t < b-\varepsilon \}$ for $\varepsilon > 0$.

Expressions of $\mathbb E_b\left[e^{-q \tau^c_{b,\varepsilon}}\textbf{1}_{\{X^c_{\tau^c_{b,\varepsilon}}=b+\varepsilon\}}\right]$ and $\mathbb E_b\left[e^{-q \tau^c_{b,\varepsilon}}\textbf{1}_{\{X^c_{\tau^c_{b,\varepsilon}}=b-\varepsilon\}}\right]$ for $q > 0$ are known, one can refer to Borodin and Salminen (2002). Actually, for $\alpha =0$, applying a similar but simple discussion to Lemma 2.1, we can obtain
\begin{equation}
\mathbb E_x\left[e^{-q \tau^c_{b,\varepsilon}} L_i(X^c_{\tau^c_{b,\varepsilon}})\right]=L_i(x), \ \ b-\varepsilon \leq x \leq b+\varepsilon,
\end{equation}
where for $i=1,2$,
\[
L_i(x)=\int_{\Pi_i}|z^{\frac{q}{\kappa}-1}|e^{-\frac{\sigma^2}{4\kappa}z^2+\frac{\mu}{\kappa} z}e^{-x z}dz, \ \  b-\varepsilon \leq x \leq b+\varepsilon,
\]
with  $\Pi_1=(0,\infty)$ and $\Pi_2=(-\infty,0)$. Formula (3.16) is enough to obtain the desired result for $\alpha =0$, and the case of $\alpha \in \mathbb R$ can be treated similarly as in Lemma 2.1.
In short, it holds that
\begin{equation}
\mathbb E_b\left[e^{-q \tau^c_{b,\varepsilon}}\textbf{1}_{\{X^c_{\tau^c_{b,\varepsilon}}=b+\varepsilon\}}\right]=
\frac{L_1(\hat{b})L_2(\hat{b}-\varepsilon)-L_2(\hat{b})L_1(\hat{b}-\varepsilon)}
{L_1(\hat{b}+\varepsilon)L_2(\hat{b}-\varepsilon)-
L_2(\hat{b}+\varepsilon)L_1(\hat{b}-\varepsilon)},
\end{equation}
and
\begin{equation}
\mathbb E_b\left[e^{-q \tau^c_{b,\varepsilon}}\textbf{1}_{\{X^c_{\tau^c_{b,\varepsilon}}=b-\varepsilon\}}\right]=
\frac{L_2(\hat{b})L_1(\hat{b}+\varepsilon)-L_1(\hat{b})L_2(\hat{b}+\varepsilon)}
{L_1(\hat{b}+\varepsilon)L_2(\hat{b}-\varepsilon)-
L_2(\hat{b}+\varepsilon)L_1(\hat{b} -\varepsilon)}.
\end{equation}

For given $q > 0$, after some straightforward calculations, we arrive at
\begin{equation}
\begin{split}
&\lim_{\varepsilon \downarrow 0}\frac{\frac{1}{2}-\mathbb E_b\left[e^{-q \tau^c_{b,\varepsilon}}\textbf{1}_{\{X^c_{\tau^c_{b,\varepsilon}}=b-\varepsilon\}}\right]}{\varepsilon}=
\frac{L_1^{\prime \prime}(\hat{b})L_2(\hat{b})-L_2^{\prime \prime}(\hat{b})L_1(\hat{b})}{4\left(L_1(\hat{b})L_2^{\prime}(\hat{b})-
L_2(\hat{b})L_1^{\prime}(\hat{b})\right)},\\
&\lim_{\varepsilon \downarrow 0}\frac{\frac{1}{2}-\mathbb E_b\left[e^{-q \tau^c_{b,\varepsilon}}\textbf{1}_{\{X^c_{\tau^c_{b,\varepsilon}}=b+\varepsilon\}}\right]}{\varepsilon}=-
\frac{L_1^{\prime \prime}(\hat{b})L_2(\hat{b})-L_2^{\prime \prime}(\hat{b})L_1(\hat{b})}{4\left(L_1(\hat{b})L_2^{\prime}(\hat{b})-
L_2(\hat{b})L_1^{\prime}(\hat{b})\right)},
\end{split}
\end{equation}
and
\begin{equation}
\begin{split}
&\lim_{\varepsilon \downarrow 0}\frac{1-\mathbb E_b\left[e^{-q \tau^c_{b,\varepsilon}}\right]}{\varepsilon^2}
=\frac{L_1^{\prime \prime}(\hat{b})L_2^{\prime}(\hat{b})-L_1^{\prime}(\hat{b})L_2^{\prime \prime}(\hat{b})}{2\left(L_1(\hat{b})L_2^{\prime}(\hat{b})-
L_2(\hat{b})L_1^{\prime}(\hat{b})\right)}.
\end{split}
\end{equation}
Note that $L^{\prime}_1(\hat{b}) < 0$ and $L_2^{\prime}(\hat{b}) > 0$, thus
$L_1(\hat{b})L_2^{\prime}(\hat{b})-
L_2(\hat{b})L_1^{\prime}(\hat{b}) \neq 0$.

(ii)
It is known from (3.8) and (3.10) that both  $V^{\prime}(b-)$ and $V^{\prime}(b+)$ exist, so it is enough to establish the following identity:
\begin{equation}
\lim_{\varepsilon \downarrow 0}\frac{\left(V(b+\varepsilon)+V(b-\varepsilon)\right)/2-V(b)}{\varepsilon} = 0.
\end{equation}

By recalling (3.1) and letting $T_1$ denote the first jump time of the Poisson process $N_t$ in (2.1), we deduce
\begin{equation}
\begin{split}
V(b)
&=\mathbb E_b\left[e^{-p\int_{0}^{e(s)}\textbf{1}_{\{X_t \leq b\}}dt+q X_{e(s)}}\right]\\
&=\mathbb E_b\left[e^{-p\int_{0}^{e(s)}\textbf{1}_{\{X_t \leq b\}}dt+q X_{e(s)}}\textbf{1}_{\{e(s)> \tau^c_{b,\varepsilon}\}}\textbf{1}_{\{T_1 \leq \tau^c_{b,\varepsilon}\}}\right]\\
&\ \ +
\mathbb E_b\left[e^{-p\int_{0}^{e(s)}\textbf{1}_{\{X_t \leq b\}}dt+q X_{e(s)}}\textbf{1}_{\{e(s)> \tau^c_{b,\varepsilon}\}}\textbf{1}_{\{T_1>\tau^c_{b,\varepsilon}\}}\right]\\
&\ \ +\mathbb E_b\left[e^{-p\int_{0}^{e(s)}\textbf{1}_{\{X_t \leq b\}}dt+q X_{e(s)}}\textbf{1}_{\{e(s) \leq \tau^c_{b,\varepsilon}\}}\right].
\end{split}
\end{equation}
For any $s > 0$, recall that $e(s)$ is independent of $\tau^c_{b,\varepsilon}$ and note that
\begin{equation}
\begin{split}
&0\leq \lim_{\varepsilon \downarrow 0}\mathbb E_b\left[e^{q X_{e(s)}}\textbf{1}_{\{e(s) \leq \tau^c_{b,\varepsilon}\}}\right]/\varepsilon \\
&\leq \lim_{\varepsilon \downarrow 0} \left(\mathbb E_b\left[e^{3 q X_{e(s)}}\right]\right)^{\frac{1}{3}} \left(\mathbb E_b\left[\textbf{1}_{\{e(s) \leq  \tau^c_{b,\varepsilon}\}}\right]/\varepsilon^{3/2}\right)^{\frac{2}{3}}=0,
\end{split}
\end{equation}
where the equality follows from (3.20).

For the first term on the right-hand side of (3.22), we have
\begin{equation}
\begin{split}
&\mathbb E_b\left[e^{-p\int_{0}^{e(s)}\textbf{1}_{\{X_t \leq b\}}dt+qX_{e(s)}}\textbf{1}_{\{e(s)> \tau^c_{b,\varepsilon}\}}\textbf{1}_{\{T_1 \leq  \tau^c_{b,\varepsilon}\}}\right]\\
& \leq \left\{\begin{array}{ccc}
\mathbb E_b\left[e^{qX_{e(s)}}\textbf{1}_{\{T_1 \leq  \tau^c_{b,\varepsilon}\}}\right], & if \ \ p \geq 0,\\
\frac{s}{s+p}\mathbb E_b\left[e^{q X_{e(s+p)}}\textbf{1}_{\{T_1 \leq  \tau^c_{b,\varepsilon}\}}\right],
& if \ \ -s < p < 0.
\end{array}\right.
\end{split}
\end{equation}
From (3.23) and (3.24), we arrive at (since $T_1$ is an exponentially distributed random and independent of $\tau^c_{b,\varepsilon}$)
\begin{equation}
\begin{split}
&0 \leq \lim_{\varepsilon \downarrow 0}\mathbb E_b\left[e^{-p\int_{0}^{e(s)}\textbf{1}_{\{X_t \leq b\}}dt+qX_{e(s)}}\textbf{1}_{\{e(s)> \tau^c_{b,\varepsilon}\}}\textbf{1}_{\{T_1 \leq  \tau^c_{b,\varepsilon}\}}\right]/\varepsilon =0.
\end{split}
\end{equation}
Similar calculations show that
\begin{equation}
\begin{split}
&0\leq \lim_{\varepsilon \downarrow 0}\mathbb E_b\left[e^{-p\int_{0}^{e(s)}\textbf{1}_{\{X_t \leq b\}}dt+q X_{e(s)}}\textbf{1}_{\{e(s) \leq \tau^c_{b,\varepsilon}\}}\right]/\varepsilon \leq 0.
\end{split}
\end{equation}

Note that $\{X_t,t < T_1\}$ and $\{X^c_t,t < T_1\}$ have the same distribution. This fact and the application of the strong Markov property of $X$ will yield
\begin{equation}
\begin{split}
&\mathbb E_b\left[e^{-p\int_{0}^{e(s)}\textbf{1}_{\{X_t \leq b\}}dt+qX_{e(s)}}\textbf{1}_{\{e(s) > \tau^c_{b,\varepsilon}\}}\textbf{1}_{\{T_1 > \tau^c_{b,\varepsilon}\}}\right]\\
&=\mathbb E_b\left[e^{-p\int_{0}^{\tau^c_{b,\varepsilon}}\textbf{1}_{\{X^c_t \leq b\}}dt}
e^{-s\tau^c_{b,\varepsilon}}\textbf{1}_{\{T_1>\tau_{b,\varepsilon}^c\}}V(X^c_{\tau_{b,\varepsilon}^c})\right]\\
&=\mathbb E_b\left[e^{-p\int_{0}^{\tau^c_{b,\varepsilon}}\textbf{1}_{\{X^c_t \leq b\}}dt}
e^{-(s+\lambda)\tau_{b,\varepsilon}^c}
\textbf{1}_{\{X^c_{\tau_{b,\varepsilon}^c}=b+\varepsilon\}}\right]V(b+\varepsilon)\\
&\ \ +\mathbb E_b\left[e^{-p\int_{0}^{\tau_{b,\varepsilon}^c}\textbf{1}_{\{X^c_t \leq b\}}dt}
e^{-(s+\lambda)\tau_{b,\varepsilon}^c}
\textbf{1}_{\{X^c_{\tau_{b,\varepsilon}^c}=b-\varepsilon\}}\right]V(b-\varepsilon),
\end{split}
\end{equation}
where the second equality is due to the fact that $T_1$ is independent of $X^c$ and $\tau_{b,\varepsilon}^c$.
For the sake of brevity, the two items on the right-hand side of (3.27) are denoted respectively by $T_{b+\varepsilon}\cdot V(b+\varepsilon)$ and $T_{b-\varepsilon}\cdot V(b-\varepsilon)$. It is clear that
\begin{equation}
e^{-(s+\lambda+\max\{p,0\})\tau_{b,\varepsilon}^c} \leq e^{-p\int_{0}^{\tau_{b,\varepsilon}^c}\textbf{1}_{\{X^c_s \leq b\}}ds} e^{-(s+\lambda)\tau_{b,\varepsilon}^c} \leq e^{-(s+\lambda+\min\{p,0\})\tau_{b,\varepsilon}^c}.
\end{equation}
The last formula and (3.19) imply that
\begin{equation}
\begin{split}
&-\infty < \underline{\lim}_{\varepsilon \downarrow 0}\frac{\frac{1}{2}-T_{b+\varepsilon}}{\varepsilon} \leq \overline{\lim}_{\varepsilon \downarrow 0}\frac{\frac{1}{2}-T_{b+\varepsilon}}{\varepsilon}< \infty,\\
&-\infty < \underline{\lim}_{\varepsilon \downarrow 0}\frac{\frac{1}{2}-T_{b-\varepsilon}}{\varepsilon} \leq \overline{\lim}_{\varepsilon \downarrow 0}\frac{\frac{1}{2}-T_{b-\varepsilon}}{\varepsilon}< \infty,
\end{split}
\end{equation}
and
\begin{equation}
\begin{split}
\lim_{\varepsilon \downarrow 0}\frac{1-T_{b-\varepsilon}-T_{b+\varepsilon}}{\varepsilon}=0.
\end{split}
\end{equation}

Therefore, from (3.22), (3.25), (3.26) and (3.27), we obtain the desired conclusion that
\begin{equation}
\begin{split}
&\lim_{\varepsilon \downarrow 0}\frac{\left(V(b+\varepsilon)+V(b-\varepsilon)\right)/2-V(b)}{\varepsilon}\\
&=\lim_{\varepsilon \downarrow 0}\left(V(b+\varepsilon)\frac{1/2-T_{b+\varepsilon}}{\varepsilon}+V(b-\varepsilon)
\frac{1/2-T_{b-\varepsilon}}{\varepsilon}\right)\\
&=\lim_{\varepsilon \downarrow 0}\left\{(V(b+\varepsilon)-V(b))\frac{1/2-T_{b+\varepsilon}}{\varepsilon}+(V(b-\varepsilon)-V(b))
\frac{1/2-T_{b-\varepsilon}}{\varepsilon}\right\}\\
&=0.
\end{split}
\end{equation}
where the second equality follows from (3.30) and the third one is due to (3.29) and the result that $V(x)$ is continuous at $b$ (see (3.8) and (3.10) and (3.12)).
\end{proof}

\begin{Remark}
The number $3$ appeared in the restriction of $-\vartheta < 3\cdot q < \eta$ is not important. In fact, we use the number $3$ only in the derivation of (3.23). To guarantee $\lim_{\varepsilon \downarrow 0}\mathbb E_b\left[e^{q X_{e(s)}}\rm{\bf{1}}_{\{e(s) \leq \tau^c_{b,\varepsilon}\}}\right]/\varepsilon=0$, it is enough to require that $-\vartheta < \beta \cdot q < \eta$ for some $\beta > 2$ so that $\beta_1:=\frac{\beta}{\beta-1} < 2$, which ensures that (recall (3.20))
\[
\lim_{\varepsilon \downarrow 0}  \mathbb E_b\left[\rm{\bf{1}}_{\{e(s) \leq  \tau^c_{b,\varepsilon}\}}\right]/\varepsilon^{\beta_1}=0.
\]
\end{Remark}

\begin{Remark}
The distribution of $Y_1$ in (2.1) has no influence on the derivation of (3.21), this means that formula (3.21) holds for a process $X$ given by (1.1) and (2.1) with arbitrary jump distributions.
\end{Remark}


\end{document}